\documentclass[12pt]{article}
\usepackage{amsmath,amsfonts,amssymb,amsthm,url}
\usepackage{graphicx,epsf,epsfig}
\usepackage{pdfpages}
\usepackage{booktabs,caption}
\usepackage[flushleft]{threeparttable}
\usepackage{enumerate}
\usepackage{natbib}
\usepackage{url} 
\newtheorem{theorem}{Theorem}
\newtheorem{condition}{Condition}
\newtheorem{proposition}{Proposition}
\newtheorem{lemma}{Lemma}
\newtheorem{remark}{Remark}
\newtheorem{corollary}{Corollary}
\newcommand{\blind}{1}

\newcommand{\T}{{\rm T}}

\makeatletter
\newcommand*{\indep}{%
  \mathbin{%
    \mathpalette{\@indep}{}%
  }%
}
\newcommand*{\nindep}{%
  \mathbin{
    \mathpalette{\@indep}{\not}
  }%
}
\newcommand*{\@indep}[2]{%
  \sbox0{$#1\perp\m@th$}
  \sbox2{$#1=$}
  \sbox4{$#1\vcenter{}$}
  \rlap{\copy0}
  \dimen@=\dimexpr\ht2-\ht4-.2pt\relax
  \kern\dimen@
  {#2}%
  \kern\dimen@
  \copy0 
} 
\makeatother

\def\T{{ \mathrm{\scriptscriptstyle T} }}

\def\d{{\rm d}}
\def\param{{\theta,\eta}}
\def\paraml{{\theta,\Lambda}}

\def\<{\langle}
\def\>{\rangle}

\begin{document}

\def\spacingset#1{\renewcommand{\baselinestretch}%
{#1}\small\normalsize} \spacingset{1}


\if1\blind
{
  \title{\bf A unified approach to calculation of information operators in semiparametric models}

  \author{Lu Mao\thanks{
    lmao@biostat.wisc.edu}\hspace{.2cm}\\
    Department of Biostatistics and Medical Informatics, \\School of Medicine and Public Health,
		\\ University of Wisconsin-Madison, Madison, WI, USA.\\
   }\date{}
  \maketitle
} \fi

\if0\blind
{
  \bigskip
  \bigskip
  \bigskip
  \begin{center}
    {\LARGE\bf On Causal Estimation Using $U$-Statistics}
\end{center}
  \medskip
} \fi

\bigskip
%
\spacingset{1.45} 

\begin{abstract}
The infinite-dimensional information operator for the nuisance parameter plays a key role in semiparametric inference,
as it is closely related to the regular estimability of the target parameter. Calculation of information operators has
traditionally proceeded in a case-by-case manner and has easily entailed lengthy derivations with complicated arguments.
 We develop a unified framework for this task by exploiting commonality in the form of semiparametric likelihoods.
 The general formula allows one to derive information operators with simple calculus and, if necessary
at all, a minimal amount of probabilistic evaluations. This streamlined approach shows 
its efficiency and versatility in application to a number of popular models in survival analysis,
 inverse problems, and missing data.
\end{abstract}


\section{Introduction}
Consider a smooth parametric model with density $p_{\theta,\psi}$, where $\theta\in\mathbb R^p$
is the parameter of interest and $\psi\in\mathbb R^q$ is a nuisance parameter. Suppose that the 
information matrix for $(\theta^\T,\psi^\T)^\T$ can be written in the following partitioned form
\begin{equation}\label{eq:par_info}
\mathcal I=\left(\begin{array}{cc}I_{\theta\theta}&I_{\theta\psi}\\
I_{\psi\theta}&I_{\psi\psi}\end{array}\right).
\end{equation} 
Then, the efficient information for $\theta$ is 
\begin{equation}\label{eq:eff_info}
I_{\theta\theta\cdot\psi}=I_{\theta\theta}-I_{\theta\psi}I_{\psi\psi}^{-1}I_{\psi\theta}.
\end{equation}
As is well known, in the presence of unknown $\psi$, 
regular and asymptotically linear estimators
for $\theta$ exist, among which the maximum likelihood estimator is the most efficient,
 if $I_{\theta\theta\cdot\psi}$ 
is positive definite \citep[Ch. 2]{bickel:1993}. 
In this paper, we call a parameter regularly estimable if a regular and asymptotically linear estimator 
exists.
Since information matrices are always non-negative definite,
here positive definiteness is equivalent to invertibility.
By rules of matrix inversion applied to \eqref{eq:eff_info}, one has that
\begin{equation}\label{eq:par_inv}
I_{\theta\theta\cdot\psi}^{-1}=I_{\theta\theta}^{-1}+I_{\theta\theta}^{-1}I_{\theta\psi}I_{\psi\psi\cdot\theta}^{-1}I_{\psi\theta} I_{\theta\theta}^{-1},
\end{equation}
where $I_{\psi\psi\cdot\theta}=I_{\psi\psi}-I_{\psi\theta} I_{\theta\theta}^{-1}I_{\theta\psi}$, 
provided that all matrix inverses involved exist. So, if $I_{\theta\theta}$ is invertible,
 then invertibility of $I_{\theta\theta\cdot\psi}$
is equivalent to that of $I_{\psi\psi\cdot\theta}$.

Similarly in semiparametric models, positivity of efficient information for the target parameter is necessary for 
its regular estimability and is usually the key condition governing the asymptotic efficiency of the maximum likelihood estimator
\citep[see, e.g., ][Ch. 25]{van:1998}. 
However, proving non-singularity of the efficient information 
in the presence of infinite-dimensional nuisance parameters
requires considerably more effort than matrix inversion. 
The difficulty arises because the information operators for the nuisance parameter 
are maps between infinite-dimensional spaces, and consequently their properties are generally more elusive than are those of matrices. 
In addition, the operators for the nuisance parameters may not be invertible at all, 
so that results analogous to \eqref{eq:par_inv} do not apply.

Consider a semiparametric model indexed by $\theta$ and an infinite-dimensional parameter $\eta$:
\begin{equation}\label{eq:main_model}
\{P_{\theta,\eta}:\theta\in\mathbb R^p, \eta\in H\},
\end{equation}
where $H$ is a nonparametric space of probability measures or positive finite measures. 
Use $p_{\theta,\eta}$ to denote the density of $P_{\theta,\eta}$ with respect to some dominating measure. 
Let $\dot l_\param$ denote the score function for $\theta$ and $B_\param: \dot H_\eta \to L_2(P_\param)$ the score operator for $\eta$,
where $\dot H_\eta \subset L_2(\eta)$ is the original tangent space for $\eta$ \citep[see, e.g.,][]{bickel:1993} and
$L_2(\mu)$ denotes the space of all $\mu$-square-integrable functions. If $\eta$ is a probability measure, then
$\dot H_\eta =L_2^0(\eta)$, the space of all $\eta$-mean zero square-integrable functions; if $\eta$ is a positive
finite measure, then $\dot H_\eta =L_2(\eta)$. In practice, one can work with a smaller set than $\dot H_\eta$,
e.g., the subset of all bounded functions with bounded variation \citep[see, e.g., ][Ch. 25]{van:1998}. 
In such cases, the score functions for $\eta$ can typically be generated by taking $B_\param a=\partial\log p_{\theta,\eta_t}/\partial t|_{t=0}$
with $\d\eta_t=(1+ta)\d\eta$.

Let $I_\param=E_\param(\dot l_\param^{\otimes 2}) $ denote the information matrix for $\theta$ had $\eta$ been known, 
where $v^{\otimes 2}=vv^\T$ for any vector $v$.
 Let $B_\param^*: 
L_2(P_\param)\to \dot H_\eta$ denote the adjoint of $B_\param$. The information operator for $(\param)$ can be
expressed in a form analogous to \eqref{eq:par_info}:
\begin{equation}\label{eq:semi_info}
\mathcal I_\param=\left\{\begin{array}{cc}I_\param&P_\param (\dot l_\param B_\param\cdot)\\
 B_\param^*\dot l_\param^\T &B_\param^*B_\param\cdot\end{array}\right\},
\end{equation} 
which acts upon $\mathbb R^p\times\dot H_\eta$. Here and after, operations on a vector with components in a
Hilbert space are understood to operate component-wise. Denote the efficient information for $\theta$
by $\tilde I_\param$. Provided that $I_\param$ is non-singular, define $V_\param: \dot H_\eta\to\dot H_\eta$ by
$V_\param=B_\param^*B_\param+K_\param$, where $K_\param=-B_\param^*\dot l_\param^\T I_\param^{-1}P_\param (\dot l_\param B_\param)$.
The operator $V_\param$ is the efficient information operator for $\eta$ in the presence of unknown $\theta$ and 
is the semiparametric analog of $I_{\psi\psi\cdot\theta}$ in \eqref{eq:par_inv}. Similar to the parametric case, 
$\tilde I_\param$ is non-singular if $V_\param: \dot H_\eta\to\dot H_\eta$ is continuously invertible, which means that
the operator has a continuous inverse. The additional continuity requirement for the inverse is automatic for matrices
but constitutes a condition for infinite-dimensional operators. 
Intuitively, continuous invertibility of $V_\param$ implies that $\theta$ and $\eta$ are not
locally confounded. Unlike the parametric case, however, continuous invertibility of $V_\param$ is generally
not straightforward to prove and may in fact not be true because
$B_\param^*B_\param$ may not be continuously invertible. The latter scenario has the serious consequence that
some aspects of $\eta$ are not estimable at the standard $n^{-1/2}$
rate. However, that does not mean that $\theta$ is necessarily not regularly estimable either. In fact,
for models suitably smooth and identifiable in $\theta$, there usually exists an alternative route to prove 
the positive definiteness of its efficient information.

Depending on whether $B_\param^*B_\param$ is continuously invertible, many of the semiparametric models in the literature
can be classified into one of the following two categories.
\begin{description}
\item[] Category 1: $B_\param^*B_\param$ can be written as the sum of a continuously invertible operator $A_\param$
and a compact operator $C_\param$, one that maps the unit ball of $\dot H_\eta$ into a totally bounded set. Because
$K_\param$ is also a compact operator, by Fredholm theory \citep{rudin:1973}, $V_\param$ is continuously invertible if it is one-to-one. 
The latter can usually be proved through local identifiability arguments. 
If the estimator is obtained by the maximum likelihood, its asymptotic properties are best handled by the Likelihood Equations
approach \citep[see \S 25.12 of][]{van:1998}.
Examples include \citet{murphy:1995}, \citet{murphy:1997}, \citet{parner:1998}, \citet{kosorok:2004}, \citet{zeng:2006}, and \citet{mao:2017}, among others.

\item[] Category 2: $B_\param^*B_\param$ is not invertible and is in fact in the form of a compact integral operator. 
For such cases,
the above approach by inverting the joint information does not work. Instead, one seeks to derive, or at least show existence
of, a least favorable direction $\tilde a$ satisfying the normal equation
\begin{equation}\label{eq:normal}
B_\param^*B_\param\tilde a=B_\param^*\dot l_\param,\hspace{5mm}\eta\mbox{-almost everywhere}.
\end{equation}
Then, the efficient score for $\theta$, defined as the projection of $\dot l_\param$ onto the orthogonal complement of 
$\{B_\param a: a\in \dot H_\eta\}$, is $\tilde l_\param=\dot l_\param-B_\param\tilde a$.
This is because  $\<\dot l_\param-B_\param\tilde a, B_\param a\>_{P_\param}=\<B_\param^*(\dot l_\param-B_\param\tilde a), a\>_\eta=0$
for all $a\in\dot H_\eta$, where $\<\cdot, \cdot\>_\mu$ denotes the inner product in $L_2(\mu)$.
The non-singularity of $\tilde I_\param=E_\param(\tilde l_\param^{\otimes 2})$ may again be proved through local identifiability arguments.
If the estimator is obtained by the maximum likelihood, its asymptotic properties are best handled by the 
Approximately Least-Favorable Sub-models approach \citep[see \S 25.11 of][]{van:1998}.
Examples include \citet{huang:1995}, \citet{huang:1996}, \citet{huang:1997}, and \citet{zeng:2016}, among others.
\end{description}

For both scenarios, it is important that one derive the specific forms of $B_\param^*B_\param$ and $B_\param^*\dot l_\param$
and check if the corresponding requirements are met to guarantee positive information for the parameter of interest. 
Such analyses usually constitute the main steps in deriving the asymptotic properties
of the maximum likelihood estimators, and should not be taken lightly since semiparametric likelihoods may be ill-behaved \citep[\S 5.2]{van:2002}.
Calculation of the information operators has mostly been treated on
a model-by-model basis in the literature. 

In this paper, we establish a unified framework for this task based on a general form of semiparametric likelihoods.
The theory developed here allows one to bypass complicated functional analytic and probabilistic arguments which are
characteristic of individual, model-specific treatments. It also offers new insights into results obtained previously on 
seemingly ad hoc basis.

\section{The general formula}
In the parametric setting, it is well known that the information matrix
can be equivalently expressed as the negative expectation of the derivative of the score function. 
For information operators in semiparametric models, one can also exploit this equivalency to simplify calculation.
The following lemma lays the foundation for the subsequent derivation of a general formula for information operators. 
Throughout, we assume that model \eqref{eq:main_model} is sufficiently smooth to warrant point-wise differentiation as a means
of score generation and to justify 
interchange of expectation and differentiation whenever appropriate. 
For a more general set-up for smooth models based on differentiability in quadratic mean, see \citet{bickel:1993}.

\begin{lemma}\label{lem:deriv}
Let $g_\param$ be a score function for model \eqref{eq:main_model} at $(\param)$.
Write $\d\eta_t=(1+t b)\d\eta$, $b\in\dot H_\eta$. Then,
\[\<B_\param^*g_\param, b\>_\eta=E_\param\left(g_\param B_\param b\right)=-E_\param\left\{\frac{\partial }{\partial t}g_{\theta,\eta_t}\Big|_{t=0}\right\}.\]
\end{lemma}

The following theorem presents the formulas for the score functions, score operators, and information operators
based on a general form of semiparametric likelihoods. 
The proof involves straightforward application of 
Lemma \ref{lem:deriv} with $g_\param=\dot l_\param$ or $B_\param a$.
Unless otherwise specified, we use $\dot f$ and $\ddot f$ to denote the first and second derivatives of 
a generic smooth function $f$. 

\begin{theorem}\label{thm:main}
Suppose that the log-likelihood for model \eqref{eq:main_model} takes the following form:
\begin{equation}\label{eq:lik}
\log p_\param=r(\theta)+f\left\{\int g(u;\theta)\d\eta(u)\right\}+L(\log\dot\eta),
\end{equation}
where $r$, $f$, and $g$ are real-valued data-dependent functions 
and $L$ is a data-dependent linear functional on the closed linear span of the space for $\log\dot\eta$, 
the log-density of $\eta$ with respect to certain dominating measure. Write $\dot g(u;\theta)=\partial g(u;\theta)/\partial\theta$, 
$\dot f_\param=\dot f(g_\param)$, and $\ddot f_\param=\ddot f(g_\param)$, where $g_\param=\int g(u;\theta)\d\eta(u)$.
Then,  if $\dot H_\eta=L_2(\eta)$, we have that
\begin{align}\label{eq:score_info}
\dot l_\param&=\dot r(\theta)+\dot f_\param\int\dot g(u;\theta)\d\eta(u),\notag\\
B_\param a&=\dot f_\param\int g(u;\theta)a(u)\d\eta(u)+L(a),\notag\\
B_\param^*\dot l_\param(\cdot)&=\int \beta_\param(\cdot, u)\d\eta(u)+\alpha_\param(\cdot),\notag\\
B_\param^*B_\param a(\cdot)&=\gamma_\param(\cdot)a(\cdot)+\int \kappa_\param(\cdot, u) a(u)\d\eta(u),
\end{align}
where
\begin{align}\label{eq:functions}
\alpha_\param(\cdot)&=-E_\param\left\{\dot f_\param\dot g(\cdot;\theta)\right\},\hspace{5mm}
\beta_\param(\cdot, u)=-E_\param\left\{\ddot f_\param g(\cdot;\theta)\dot g(u;\theta)\right\},\notag\\
\gamma_\param(\cdot)&=-E_\param\left\{\dot f_\param g(\cdot;\theta)\right\},\hspace{5mm}
\kappa_\param(\cdot, u)=-E_\param\left\{\ddot f_\param g(\cdot;\theta)g(u;\theta)\right\}.
\end{align}
 If $\dot H_\eta=L_2^0(\eta)$, the results are the same 
except that $B_\param^*B_\param a(\cdot)$ in \eqref{eq:score_info} is replaced by
\[B_\param^*B_\param a(\cdot)=\gamma_\param(\cdot)a(\cdot)-\int\gamma_\param(u)a(u)\d\eta(u) +\int \kappa_\param(\cdot, u) a(u)\d\eta(u),\]
and the functions $\dot g(\cdot;\theta)$, $g(\cdot;\theta)$,
and $g(u;\theta)$ on the right hand sides of the equations in \eqref{eq:functions} are replaced by
$\dot g(\cdot;\theta)-\dot g_\param$, $g(\cdot;\theta)-g_\param$,
and $g(u;\theta)-g_\param$, respectively, where $\dot g_\param=\int\dot g(u;\theta)\d\eta(u)$.
\end{theorem}

\begin{remark}\label{rem:extend}
For notational simplicity, we have assumed that the function $g$ in Theorem \ref{thm:main} is real-valued. 
It is straightforward to extend the results to the case with vector-valued
$g$. Furthermore, instead of a single nuisance parameter $\eta$, 
one can extend the framework to accommodate multiple nuisance parameters $\eta=(\eta_1,\ldots,\eta_K)^\T$. In such cases, 
the original tangent space for $\eta$ will be $\dot H_1\times\ldots\times\dot H_K$, where $\dot H_k$ is the 
original tangent space for $\eta_k$ $(k=1,\ldots, K)$. Such extensions will be considered and illustrated in \S \ref{sec:rec}.
\end{remark}

Under the condition of Theorem \ref{thm:main}, the information operator for $\eta$ can be written as
the sum of a multiplication operator with multiplier $\gamma_\param$
and a compact Hilbert-Schmidt integral operator with kernel $k_\param$, insofar as  
$k_\param$ is square-integrable by $\eta\times\eta$.
The multiplication operator is continuously invertible if
$\gamma_\param(\cdot)$ is bounded above and away from zero. If so,
 the model is of Category 1.
Likewise, if $\gamma_\param\equiv 0$, then it is of Category 2.
The local identifiability condition needed for both categories to ensure non-singularity of $\tilde I_\param$ can be stated formally
as follows.
\begin{condition}[Local identifiability]\label{cond:local_id}
If 
\begin{equation}\label{eq:local_id}
h^\T\dot r(\theta)+\dot f_\param\int h^\T\dot g(u;\theta)\d\eta(u)+\dot f_\param\int g(u;\theta)a(u)\d\eta(u)+L(a)=0
\end{equation}
$P_\param$-almost surely for some $h\in\mathbb R^p$ and $a\in\dot H_\eta$, then $h=0$ and $a(\cdot)\equiv 0$.
\end{condition}
Since the left hand side of \eqref{eq:local_id} is a score function in the general form $h^\T\dot l_\param+B_\param a$,
Condition \ref{cond:local_id} simply says that the joint score operator is one-to-one so that local alternatives to $(\param)$
in all possible directions can be identified. In particular, it implies that $I_\param$ is positive definite.
To use it to show that $V_\param$ is one-to-one for Category 1 problems, one may take $h=I_\param^{-1}P_\param(\dot l_\param B_\param b)$
and $a=-b$ to find that $V_\param b=0$ implies $b(\cdot)\equiv 0$.
For Category 2 problems, one may take $a=-h^\T\tilde a$ to show that $\tilde I_\param$ is positive definite
provided that $\tilde a$ as a solution to \eqref{eq:normal} exists. 
\begin{corollary}\label{cor:pd}
Suppose that Condition \ref{cond:local_id} is satisfied.
Then, $\tilde I_\param$ is positive definite if either of the following is true:
\begin{description}
\item[](1) There exist $M>0$ such that $M^{-1}\leq \gamma_\param(\cdot)\leq M$, or
\item[](2) $\gamma_\param\equiv 0$ and the solution $\tilde a$ to
\begin{equation}\label{eq:c2_normal}
\int\kappa_\param(\cdot, u)\tilde a(u)\d\eta(u)=B_\param^*\dot l_\param(\cdot)
\end{equation}
exists.
\end{description}
\end{corollary}
In the second case, solution of the least favorable direction $\tilde a$ usually starts with taking derivatives on both sides of \eqref{eq:c2_normal}.
For example, \citet{huang:1997} took this route to show that the solution exists for the Cox model with case-2 interval-censored data.
In particular, this approach requires that $B_\param^*\dot l_\param(\cdot)$ be a smooth function and lie in the range of $B_\param B_\param^*$.

The following two propositions can usually simply calculations for Category 2 problems. The first one is fairly intuitive:
if the density of $\eta$ does not appear in the likelihood, 
then information on some aspects thereof cannot be recovered in the first order.
Thus, one expects $B_\param^*B_\param$ to be not continuously invertible.
\begin{proposition}\label{prop:c2}
Under the conditions of Theorem \ref{thm:main}, if $L\equiv 0$, then $\gamma_\param(\cdot)=0$ $\eta$-almost everywhere.
\end{proposition}
\begin{proof}
With $L\equiv 0$, use $E_\param(B_\param a)=0$ to find that $\int \gamma_\param a\d\eta$ for all $a \in\dot H_\eta$, implying $\gamma_\param(\cdot)=0$
$\eta$-almost everywhere.
\end{proof}

For Category 2 problems, derivation of the normal equation \eqref{eq:c2_normal} can be further simplified if $p_\param$ is a conditional
density in certain form.
\begin{proposition}\label{prop:cond}
Suppose that the density for model \eqref{eq:main_model} can be written in the following form:
\begin{equation}
p_\param(\mathcal O_1,\mathcal O_2)=p_\param^{(1\mid 2)}(\mathcal O_1\mid\mathcal O_2)p^{(2)}(\mathcal O_2),
\end{equation}
where $p_\param^{(1\mid 2)}(\cdot\mid\cdot)$ is the conditional density of $\mathcal O_1$ given $\mathcal O_2$
and $p^{(2)}(\cdot)$ is the marginal density of $\mathcal O_2$. If the log-likelihood $\log p_\param^{(1\mid 2)}(\mathcal O_1\mid\mathcal O_2)$ can be written in the form of \eqref{eq:lik} with $r(\cdot)\equiv 0$, $L\equiv 0$, 
$f(\cdot)=\tilde f(\cdot ,\mathcal O_1)$ and $g(\cdot;\theta)=\tilde g(\cdot,\mathcal O_2;\theta)$ for some deterministic functions
$\tilde f$ and $\tilde g$, then
$\alpha_\param(\cdot)\equiv 0$ and $\gamma_\param(\cdot)\equiv 0$. Hence, the normal equation \eqref{eq:c2_normal} becomes
\begin{equation}\label{eq:normal1}
\int \kappa_\param(\cdot, u) a(u)\d\eta(u)=\int \beta_\param(\cdot, u)\d\eta(u).
\end{equation}
\end{proposition}
\begin{proof}
In light of Proposition \ref{prop:c2}, we only need to show that $\alpha_\param(\cdot)\equiv 0$. Because $\dot f_\param$
is now a score function for the conditional density of $\mathcal O_1$ given $\mathcal O_2$, we have that
$E_\param(\dot f_\param\mid\mathcal O_2)=0$. The result follows from the fact that $\dot g(\cdot;\theta)$ depends on $\mathcal O_2$ only.
\end{proof}
Proposition \ref{prop:cond} applies to all standard regression models with interval-censored data where the examination 
times are conditionally independent of the event times given covariates \citep[see, e.g., ][]{sun:2007}. Indeed, let $T$ be the event time of interest, $U$
be the sequence of examination times, $\delta(T, U)$ be the observed indicators for the affiliation of $T$ to the intervals partitioned by $U$,
and $Z$ be the covariates. 
If $T\indep U\mid Z$ and  $(\param)$ parametrizes 
only the conditional distribution of $T$ given $Z$, the conditions of Proposition \ref{prop:cond} are satisfied with 
$\mathcal O_1=\delta(T, U)$ and $\mathcal O_2=(U, Z)$.

Finally, we consider a nonparametric model $\{P_\eta: \eta\in H\}$ as a special case of \eqref{eq:main_model}. Here,
one is interested in, $\chi(\eta)$, a functional of $\eta$, with pathwise derivative $\dot\chi(\eta)$. Then, the functional
$\chi(\eta)$ is regularly estimable under $P_\eta$ if a solution $\tilde a$ to the normal equation
\begin{equation}\label{eq:np_normal}
B_\eta^*B_\eta\tilde a=\dot\chi(\eta)
\end{equation}
exists, where $B_\eta$ is the score operator for $\eta$.
Then, the efficient influence function is $B_\param\tilde a$.
The score and information operators can be calculated similarly to Theorem \ref{thm:main}.
\begin{corollary}\label{cor:np}
Suppose that the log-likelihood for model $\{P_\eta: \eta\in H\}$ takes the following form:
\begin{equation}\label{eq:np_lik}
\log p_\param=f\left\{\int g(u)\d\eta(u)\right\}+L(\log\dot\eta),
\end{equation}
where $f$, $g$, and $L$ are data-dependent functions defined analogously to their counterparts in Theorem \ref{thm:main}.
 Write 
$\dot f_\eta=\dot f(g_\eta)$, and $\ddot f_\eta=\ddot f(g_\eta)$, where $g_\eta=\int g(u)\d\eta(u)$.
Then, if $\dot H_\eta=L_2(\eta)$, we have that
\begin{align*}
B_\eta a&=\dot f_\eta\int g(u)a(u)\d\eta(u)+L(a),\\
B_\eta^*B_\eta a(\cdot)&=\gamma_\eta(\cdot)a(\cdot)+\int \kappa_\eta(\cdot, u) a(u)\d\eta(u).\hspace{5mm}a\in\dot H_\eta,\\
\end{align*}
where 
\begin{align}\label{eq:np_functions}
\gamma_\eta(\cdot)=-E_\eta\left\{\dot f_\eta g(\cdot)\right\},\hspace{5mm}
\kappa_\eta(\cdot, u)=-E_\eta\left\{\ddot f_\eta g(\cdot)g(u)\right\}.
\end{align}
If $\dot H_\eta=L_2^0(\eta)$, then $B_\eta^*B_\eta a(\cdot)$ is replaced by
\[B_\eta^*B_\eta a(\cdot)=\gamma_\eta(\cdot)a(\cdot)-\int\gamma_\eta(u)a(u)\d\eta(u)+\int \kappa_\eta(\cdot, u) a(u)\d\eta(u),\]
and $g(\cdot)$ and $g(u)$ on the right hand sides of the equations in \eqref{eq:np_functions} are replaced by
 $g(\cdot)-g_\eta$,
and $g(u)-g_\eta$, respectively. Furthermore, if $L\equiv 0$, then $\gamma_\eta(\cdot)=0$
$\eta$-almost everywhere.
\end{corollary}

\section{Applications}\label{sec:app}
\subsection{The Cox model under right- and interval-censorships}\label{sec:cox}
First, consider the Cox model with right-censored data \citep{cox:1975}, the archetype of semiparametric inference.
Let $T$ denote the event time of interest and $Z$ a vector of covariates. The Cox proportional hazards model
specifies that
\begin{equation}\label{eq:cox}
\Lambda(t\mid Z)=\int_0^t\exp(\theta^\T Z)\d\Lambda(s), \hspace{3mm}t\in[0,\tau],
\end{equation}
where $\Lambda(t\mid Z)$ is the conditional cumulative hazard function of $T$ given $Z$, $\theta$
is the regression parameter, $\Lambda(\cdot)$ is the baseline cumulative hazard function,
and $\tau$ is the maximum length of follow-up. Here, $\theta$ is the parameter of interest and $\eta=\Lambda$
is the nuisance parameter. Let $C$ denote the censoring time. Then, the observed data consists of
$\{\delta\equiv I(T\leq C), X\equiv T\wedge C, Z\}$, where $I(\cdot)$ is the indicator function
and $a\wedge b=\min(a, b)$.
The log-likelihood for the observed data is
\begin{equation}\label{eq:cox_lik}
\log p_{\theta,\Lambda}=\delta\log\lambda(X)+\delta\theta^\T Z-\int I(X\geq u)\exp(\theta^T Z)\d\Lambda(u),
\end{equation}
with $\lambda=\dot\Lambda$. Comparing \eqref{eq:cox_lik} with \eqref{eq:lik}, one readily recognizes
that $r(\theta)=\delta\theta^\T Z$, $g(u;\theta)=I(X\geq u)\exp(\theta^\T Z)$,
$f(x)=-x$, and $L(\log\lambda)=\delta\log\lambda(X)$. The last identity means that 
$L$ operates on $\log\lambda$ by evaluating it at $X$ and then multiplying it by $\delta$. Hence,
$\dot r(\theta)=\delta Z$, $\dot g(\cdot;\theta)=Z\exp(\theta^\T Z)I(X\geq\cdot)$, $\dot f(x)=-1$,
and $\ddot f(x)= 0$. By Theorem \ref{thm:main}, we have that $\dot l_{\theta,\Lambda}=\delta Z-\int Z\exp(\theta^\T Z)I(X\geq u)d\Lambda(u)$, $B_\paraml a=\delta a(X)-\int a(u)\exp(\theta^\T Z)I(X\geq u)d\Lambda(u)$,
$B_\paraml^*\dot l_\paraml(\cdot)=E_\paraml\{Z\exp(\theta^\T Z)I(X\geq\cdot)\}$,
and $B_\paraml^*B_\paraml a(\cdot)=\gamma_\paraml(\cdot)a(\cdot)$
with $\gamma_\paraml(\cdot)=E_\paraml\{\exp(\theta^\T Z)I(X\geq\cdot)\}$.
If $Z$ has bounded support and ${\rm pr}(X\geq\tau)>0$, we have that the multiplier 
$\gamma_\paraml(\cdot)$ is bounded above and away from zero.  
It is thus a Category 1 problem, but is special in that the efficient score can be constructed explicitly.
Indeed, the normal equation \eqref{eq:normal} can be solved with
\[\tilde a(\cdot)=\frac{B_\paraml^*\dot l_\paraml(\cdot)}{\gamma_\paraml(\cdot)}=\frac{E_\paraml\{Z\exp(\theta^\T Z)I(X\geq\cdot)\}}{E_\paraml\{\exp(\theta^\T Z)I(X\geq\cdot)\}}.\]
Then, an approximation to the efficient score $\tilde l_\paraml=\dot l_\paraml-B_\paraml\tilde a$ can be constructed
by replacing the expectations in $\tilde a$ with their empirical version, leading to the familiar partial likelihood
 score function for $\theta$ \citep{cox:1975}. Furthermore, under linear independence of $Z$, it is easy to show that Condition \ref{cond:local_id}
is satisfied so that the efficient information is positive definite.

The Cox model under case-1 interval censoring, studied in detailed by Huang (1996), offers an example in Category 2.
The conditional hazard of $T$ given $Z$ is specified by the same model \eqref{eq:cox}, but the observed data now consist of
$\{\delta\equiv I(T\leq U), U, Z\}$, where $U$ is the examination time satisfying $T\indep U\mid Z$. Clearly, the likelihood
for the observed data satisfies the conditions of Proposition \ref{prop:cond} with $\mathcal O_1=\delta$
and $\mathcal O_2=(U, Z)$. The log-likelihood is
\[\log p_{\theta,\Lambda}=\delta\log\left[1-\exp\left\{-\int_0^U\exp(\theta^\T Z)\d\Lambda(u)\right\}\right]-(1-\delta)\int_0^U\exp(\theta^\T Z)\d\Lambda(u).\]
So, we may set $g(\cdot;\theta)=\exp(\theta^\T Z)I(U\geq\cdot)$ and $f(x)=\delta\log\{1-\exp(-x)\}-(1-\delta)x$, so that
$\dot g(\cdot;\theta)=Z\exp(\theta^\T Z)I(U\geq\cdot)$, $\dot f(x)=\delta\exp(-x)/\{1-\exp(-x)\}-(1-\delta)$,
and $\ddot f(x)=-\delta\exp(-x)/\{1-\exp(-x)\}^2$. By Proposition \ref{prop:cond}, the normal equation is in the form
of \eqref{eq:normal1}, which, after straightforward iterated conditional expectation
applied to $\beta_\paraml$ and $\kappa_\paraml$, can be simplified to
\begin{equation}\label{eq:normal_cs}
E_\paraml\left\{s_\paraml^{(0)}(U)\int_0^U \tilde a(u)\d\Lambda(u)I(U\geq \cdot)\right\}=E_\paraml\left\{s_\paraml^{(1)}(U)\Lambda(U)I(U\geq \cdot)\right\},
\end{equation}
where 
$$s_\paraml^{(0)}(U)=E_\paraml\left\{\exp(2\theta^\T Z)O_\paraml(U,Z)\mid U\right\},$$
$$ s_\paraml^{(1)}(U)=E_\paraml\left\{Z\exp(2\theta^\T Z)O_\paraml(U,Z)\mid U\right\}, $$
$$O_\paraml(U,Z)=\exp\{-\exp(\theta^\T Z)\Lambda(U)\}/[1-\exp\{-\exp(\theta^\T Z)\Lambda(U)\}].$$
Assuming that the support of $U$ contains $[0, \tau]$, 
we can take derivative on both sides of \eqref{eq:normal_cs} to find that
\begin{equation}\label{eq:lfd_cs}
\int_0^t \tilde a(u)\d\Lambda(u)=\Lambda(t)\zeta_\paraml(t),\hspace{3mm}t\in[0,\tau],
\end{equation}
where $\zeta_\paraml(t)=s_\paraml^{(1)}(t)/s_\paraml^{(0)}(t)$. So, $\tilde a(t)=\zeta_\paraml(t)+\lambda(t)^{-1}\Lambda(t)\dot\zeta_\paraml(t)$.
Using \eqref{eq:lfd_cs}, one easily obtains the efficient score
\begin{align*}
\tilde l_\paraml&=\dot l_\paraml-B_\paraml\tilde a\\
&=\dot f_\paraml Z\exp(\theta^\T Z)\Lambda(U)-\dot f_\paraml \exp(\theta^\T Z)\int_0^U \tilde a(u)\d\Lambda(u)\\
&=\dot f_\paraml\exp(\theta^\T Z)\Lambda(U)\left\{Z-\zeta_\paraml(U)\right\}.
\end{align*}
\citet{huang:1996} and \citet{van:1998, van:2002} derived the same result for the efficient score 
by orthogonal projections. However, their $\dot H_\Lambda$ is specified as the space of differences in $\Lambda$
 so that their least favorable direction is essentially $\tilde\Lambda(\cdot)=\int_0^\cdot \tilde a(u)\d\Lambda(u)$.
With this comes the need to construct an Approximately Least-Favorable Sub-model for the maximum likelihood estimator 
$\hat\Lambda$ as its perturbed form $\hat\Lambda+h^\T\tilde\Lambda$ $(h\in\mathbb R^p)$ need not be
a non-decreasing function with $\hat\Lambda$ being a step function \citep[\S 25.11]{van:1998}. 
Such construction is not necessary in our approach as the perturbed form
$\int_0^\cdot(1+h^\T\tilde a)\d\hat\Lambda$ is always non-decreasing for $h$ sufficiently close to zero provided that
regularity conditions are in place to guarantee that $\tilde a$ is bounded and of bounded variation.

\subsection{Transformation models for recurrent events}\label{sec:rec}
Consider 
a recurrent event regression model studied by \citet{zeng:2006}. Let $N^*(t)$ count the number of events by time $t$ and
let $Z(t)$ denote the time-dependent left-continuous covariate processes, 
$t\in[0,\tau]$. Let $\mathcal F_t=\{N(u), Z(u+):0\leq u\leq t\}$
denote the history of the subject up to $t$. The model specifies that the cumulative intensity function of $N^*(t)$ with respect
to the filtration $\mathcal F_t$ takes the form
\begin{equation}\label{eq:model_rec}
\Lambda(t\mid Z)=G\left[\int_0^t\exp\left\{\theta^\T Z(u)\right\}\d\Lambda(u)\right],
\end{equation}
where $G$ is a known transformation function. The choice of $G(x)=x$
yields the familiar proportional intensity of model of \citet{andersen:1982}.

Let $C$ denote the censoring time and write $N(t)=N^*(t\wedge C)$. Then, the observed data
consist of $\{N(\cdot), C, Z\}$. The log-likelihood for the observed data can be written as
\begin{equation}\label{eq:lik_rec}
\log p_{\theta,\Lambda}(X)=\int \left[\log\{\dot G_\paraml(t)\}+\theta^\T Z(t)+\log\lambda(t)\right]\d N(t)-G_\paraml(\tau).
\end{equation}
Here and for the rest of the sub-section, $F_\paraml(t)=F\left[\int_0^tI(C\geq u)\exp\left\{\theta^\T Z(u)\right\}\d\Lambda(u)\right]$ for any function $F$.
Set $r(\theta)=\theta^\T\int Z(t)\d N(t)$ and $L(\log\lambda)=\int\log\lambda(t)\d N(t)$. Here, we will use a slight modification
of Theorem \ref{thm:main} by letting the function $g$ be further indexed by $t$, that is, $g(u; \theta)=\{g_t(u; \theta):t\in[0,\tau]\}$,
where $g_t(u; \theta)=I(u\leq C\wedge t )\exp\{\theta^\T Z(u)\}$.
Then, for $x=\{x_t:t\in [0, \tau]\}$, we set $f(x)=\int \log\{\dot G(x_t)\}\d N(t)-G(x_\tau)$.
Provided that $N(\tau)<\infty$, the function $g_t(u; \theta)$ depends on $t$ only on a finite set of points and is thus essentially vector-valued
as discussed in Remark \ref{rem:extend}. Calculation of the quantities in \eqref{eq:functions} 
then proceeds by rules of matrix multiplication. For example,
 $\dot f_\param g(\cdot;\theta)$
is essentially a matrix product between a row vector and a column vector and can be conveniently represented as
\[\int H_\paraml(t)g_t(\cdot; \theta)\d N(t)-\dot G_\paraml(\tau)g_\tau(\cdot; \theta),\]
where $H=\ddot G/\dot G$.
So,
\begin{align*}
\gamma_\paraml(\cdot)&=-E_\paraml\left\{\int H_\paraml(t)g_t(\cdot; \theta)\d N(t)\right\}+E_\paraml\left\{\dot G_\paraml(\tau)g_\tau(\cdot; \theta)\right\}\\
&=-E_\paraml\left[Y(\cdot)\exp\{\theta^\T Z(\cdot)\}\int_\cdot^\tau H_\param(t)\d N(t)\right]+E_\paraml\left[Y(\cdot)\exp\{\theta^\T Z(\cdot)\}\dot G_\paraml(\tau)\right]\\
&=E_\paraml\left[Y(\cdot)\exp\{\theta^\T Z(\cdot)\}\dot G_\paraml(\cdot)\right],
\end{align*}
where $Y(\cdot)=I(C\geq\cdot)$ and 
the last equality follows by the martingale property of $N(\cdot)$ and the model specification \eqref{eq:model_rec}.
If $Z$ has bounded support, $\dot G(\cdot)$ is continuous and strictly positive, and ${\rm pr}(C\geq\tau)>0$,
we have that $\gamma_\paraml(\cdot)$ is bounded above and away from zero. If, in addition, the covariate process $Z(\cdot)$
is linearly independent with probability one, then Condition \ref{cond:local_id} is satisfied and so the efficient
information for $\theta$ is positive definite by Corollary \ref{cor:pd}.

In the Supplementary Material, we consider a related but more involved example with
semiparametric regression for competing risks data (Mao \& Lin, 2017). The calculation therein extends Theorem \ref{thm:main} to multiple nuisance 
parameters as alluded to in Remark \ref{rem:extend}.

\subsection{Nonparametric models}
We first consider a standard example of one-sample right-censored data. 
Let the event time of interest be $T$ with cumulative hazard function $\Lambda$.
Let $C$ denote the censoring time with $C\indep T$. The observed data consist of
$\{X\equiv T\wedge C, \delta\equiv I(T\leq C)\}$. The goal is to estimate the survival
function $S(t)=\exp\{-\Lambda(t)\}\equiv\chi_t(\Lambda)$. By straightforward calculation, one finds that
the pathwise derivative is $\dot\chi_t(\Lambda)(\cdot)=-S(t)I(\cdot\leq t)$.

Using Corollary \ref{cor:np} and by calculations similar to those in \S \ref{sec:cox}, we can 
easily obtain that $B_\param a=\int a(u)\d M_\Lambda(u)$, where $M_\Lambda(s)=I(X\leq s)-\int_0^s I(X\geq u)\d\Lambda(u)$, 
and $B_\eta^*B_\eta a(\cdot)=\pi(\cdot)a(\cdot)$, where $\pi(\cdot)={\rm pr}(X\geq\cdot)$. Solution to the normal equation \eqref{eq:np_normal}
gives $\tilde a_t(\cdot)=-S(t)\pi(\cdot)^{-1}I(\cdot\leq t)$. Hence the efficient influence function for $S(t)$
is
\[B_\param\tilde a_t=-S(t)\int_0^t\pi(u)^{-1}\d M_\Lambda(u).\]
This is precisely the influence function, or influence curve if viewed as a process indexed by $t$ \citep{vanderlaan:2003}, of the 
familiar Kaplan--Meier estimator \citep[Ch. 6]{fleming:1991}. This result reaffirms the well-known semiparametric efficiency of
the Kaplan--Meier estimator.

A distinct class of nonparametric models involves a distribution function $\eta$, but,
 unlike the above example, data conforming to distribution $\eta$ are never directly observed. 
 Available is only a coarsened version of such observations with density $p_\eta$. 
The goal is to recover
$\eta$ from a random sample of such coarsened observations. This type of problems are called inverse problems 
\citep{hasminskii:1983, groeneboom:1992, groeneboom:2014}.

As an example, consider a mixture model for observed data $X$, whose density with respect to a dominating measure $\nu$ 
conditioning on a latent variable
$Z$ is a known function $p(x\mid z)$. 
Suppose that $Z\sim \eta$ and that one wishes to make inference on $\eta$
based on a random sample of $X$. The log-likelihood for $X$ is
\[\log p_\eta=\log \int p(X\mid z)\d\eta(z).\]
With reference to Corollary \ref{cor:np}, one has that $g(\cdot)=p(X\mid \cdot)$ and $f(x)=\log x$. So, $\ddot f(x)=-x^{-2}$.
We immediately obtain that $\gamma_\eta(\cdot)\equiv 0$ and that
\begin{align*}
\kappa_\eta(\cdot,u)&=E_\eta\left[p_\eta(X)^{-2}\{p(X\mid \cdot)-p_\eta(X)\}\{p(X\mid u)-p_\eta(X)\}\right]\\
&=\int p_\eta(x)^{-1}\{p(x\mid \cdot)-p_\eta(x)\}\{p(x\mid u)-p_\eta(x)\}\d\nu(x).
\end{align*}
Now, solution to the normal equation \eqref{eq:np_normal} depends on the specific form of the kernel $p(x\mid \cdot)$
as well as the functional of interest $\chi(\eta)$. However, it is clear that functionals such as $\eta(z)$ for
a fixed $z$ with a non-smooth pathwise derivative
of $\dot\chi(\eta)(\cdot)=I(\cdot\leq z)-\eta(z)$ are unlikely to be regularly estimable if the kernel $p(x\mid z)$
is smooth in $z$.
For a general discussion of regularly estimable functionals in this context, see \citet[\S 7]{van:1991}.

\subsection{Regression models with missing covariates}
Suppose that the conditional density of outcome $Y$ with respect to a dominating measure $\nu$ given regressor $Z$ is specified
through a parametric model $p_\theta(y\mid z)$, where $Z\sim\eta$. 
We leave the dominating measure $\nu$ arbitrary so that the set-up accommodates both categorical and continuous outcomes.
Estimation of $\theta$ is standard if $Z$ is fully observed.
In case of missing data in the regressor, however, the nonparametric component $\eta$ will get entangled with the regression parameter
and complicate inference. Problems of this type have been studied in general settings by \citet{lawless:1999}.
Here we consider a simple case with a single level of missingness in $Z$. Using the notation of \citet{tsiatis:2006},
we denote the coarsened regressor by $X=\mathcal G(Z)$, where $\mathcal G$ is a known many-to-one function.
Let $R=1$ if the full data $(Y, Z)$ are observed and $R=0$ if only the coarsened version $(Y, X)$ is available.
We assume that the data are coarsened at random, that is,
\[{\rm pr}(R=1\mid Y, Z)=\pi(Y, X),\]
where $\pi$ is some arbitrary function for the selection probability.
Provided that $\pi$ involves no aspect of $(\param)$, the log-likelihood for the observed data $\{R, Y, RZ+(1-R)X\}$ is
\[\log p_\param=R\log p_\theta(Y\mid Z)+R\log\dot\eta(Z)+(1-R)\log\int p_\theta(Y\mid z)I\{\mathcal G(z)=X\}\d\eta(z).\]
Thus, we may set $r(\theta)=R\log p_\theta(Y\mid Z)$, $L(a)=Ra(Z)$, $g(u;\theta)=p_\theta(Y\mid u)I\{\mathcal G(u)=X\}$,
and $f(s)=(1-R)\log s$. Therefore, we have that
$\dot f(s)=(1-R)s^{-1}$, $\ddot f(s)=-(1-R)s^{-2}$, and $\dot g(u;\theta)=\dot l_\theta^{\rm F}(Y\mid u)p_\theta(Y\mid u)I\{\mathcal G(u)=X\}$,
where $\dot l_\theta^{\rm F}(y\mid z)=\partial\log p_\theta(y\mid z)/\partial\theta$ is the full-data score function for $\theta$.
Using straightforward calculus, it is not hard to obtain that
\begin{align*}
\alpha_\param(\cdot)&=\int \pi\{y,\mathcal G(\cdot)\}\dot l_\theta^{\rm F}(y\mid \cdot)p_\theta(y\mid \cdot)\d\nu(y),\\
\gamma_\param(\cdot)&=\int \pi\{y,\mathcal G(\cdot)\}p_\theta(y\mid \cdot)\d\nu(y),\\
\beta_\param(\cdot, u)&=I\left\{\mathcal G(\cdot)=\mathcal G(u)\right\}\int q_\param\left\{y,\mathcal G(u)\right\}^{-1}\dot l_\theta^{\rm F}(y\mid u)
p_\theta(y\mid u)p_\theta(y\mid \cdot)\d\nu(y),\\
\kappa_\param(\cdot, u)&=I\left\{\mathcal G(\cdot)=\mathcal G(u)\right\}\int q_\param\left\{y,\mathcal G(u)\right\}^{-1}
p_\theta(y\mid u)p_\theta(y\mid \cdot)\d\nu(y),\\
\end{align*}
where $q_\param(y,x)=\int p_\theta(y\mid z)I\{\mathcal G(z)=x\}\d\eta(z)$. 
\begin{proposition}
Suppose that the following two conditions hold:
\begin{description}
\item[](a) $\pi(Y,X)\geq\delta$ for some $\delta>0$;
\item[](b) $E_\eta\{\dot l_\theta^{\rm F}(Y \mid Z)^{\otimes 2}\mid Z\}$ is positive definite almost surely.
\end{description}
Then, the efficient information $\tilde I_\param$ is positive definite.
\end{proposition}
\begin{proof}
The multiplier $\gamma_\param(\cdot)$ is clearly bounded above, and is bounded away from zero by (a). One can easily use (a) and (b) to
verify Condition \ref{cond:local_id}.
The result follows by Corollary \ref{cor:pd}.
\end{proof}

\section{Remarks}
To summarize, the proposed approach to calculation of efficient information in semiparametric models
can be streamlined in three steps.
First, write out the log-likelihood in the form of \eqref{eq:lik} and recognize the functions $r$, $g$, $f$, and $L$.
Second, with reference to Propositions \ref{prop:c2} \& {\ref{prop:cond}, calculate the needed parts for the score and information
operators according to the formulas in Theorem \ref{thm:main}. Finally, check if reasonable assumptions can be made
to satisfy the conditions of Corollary \ref{cor:pd} and thus to conclude that the efficient information is positive definite.  

The specified form of log-likelihood \eqref{eq:lik} seems to be general enough to encompass a 
surprisingly large pool of existing semiparametric models. 
It is thus reasonable to expect our framework to be amenable and useful to many new models to come.
Straightforward extensions to Theorem \ref{thm:main} exist to further expand on its applicability.
For example, the function $f$ can be made dependent on $\theta$, which will
accommodate the log-likelihoods for frailty models in survival analysis \citep[see, e.g., ][]{kosorok:2004}.
Furthermore, the conventional derivative of $g(u;\theta)$ with respect to $\theta$ can be 
replaced by a generalized
derivative, e.g., one such that $\d I(x\geq u)/\d x=I(x=u)$.
This generalization is useful when applied to the accelerated failure time model \citep{buckley:1979}, 
where $g(u;\theta)$ might be in the form of $I(X-\theta^\T Z\leq u)$. 

 Our framework is most useful when the likelihood is naturally indexed 
jointly by a Euclidean parameter and an infinite-dimensional parameter.
Other semiparametric models are more easily formulated through, say, moment or conditional moment constraints \citep[see, e.g., ][\S 6.2]{bickel:1993}. 
For such models,
it is usually easier to derive information operators via direct projection methods.

\appendix

\section*{Appendix}

\begin{proof}[Proof of Theorem \ref{thm:main}]
Calculations for $\dot l_\param$ and $B_\param a$ are straightforward. We thus focus on exhibiting the forms of $B_\param^*\dot l_\param$
and $B_\param^*B_\param a(\cdot)$. By Lemma \ref{lem:deriv}, if $\d\eta_t=(1+tb)\d\eta$, where $b$ is a bounded function with bounded variation
 in $\dot H_\eta$, then
\begin{align}\label{eq:adjoint}
\<B^*_\param\dot l_\param, b\>_\eta&=-E_\param\left(\frac{\partial}{\partial t}\dot l_{\theta,\eta_t}\right)\Bigg|_{t=0}\notag\\
&=-E_\param\left\{\ddot f_\param\int \dot g(s;\theta)\d\eta(s)\int g(u;\theta)b(u)\d\eta(u)+\dot f_\param\int \dot g(u;\theta)b(u)\d\eta(u)\right\}\notag\\
&\equiv \int c_\param(u)b(u)\d\eta(u).
\end{align}
The results for $B^*_\param\dot l_\param(\cdot)$ then follow by equating it to $c_\param(\cdot)$ if $\dot H_\eta=L_2(\eta)$
 and to $c_\param(\cdot)-\int c_\param(u)\d\eta(u)$ if $\dot H_\eta=L_2^0(\eta)$.

Now, consider $B_\param^*B_\param a_\eta(\cdot)$, where $a_\eta\in\dot H_\eta$. 
Here we have used $a_\eta$ instead of $a$ to stress the possible 
local dependence of the direction $a$ on $\eta$. 
If $\dot H_\eta=L_2(\eta)$ and $\d\eta_t=(1+tb)\d\eta$, since $b$ is bounded, we have that $\dot H_{\eta_t}=L_2(\eta)$ for all $t$.
So, $B_{\theta,\eta_t}a_\eta$ is a score function under $(\theta, \eta_t)$. Thus, $B_\param^*B_\param a_\eta(\cdot)$ can be derived 
from $\<B_\param^*B_\param a_\eta, b\>_\eta=-E_\param\left(\partial B_{\theta,\eta_t}a_\eta/\partial t\right)|_{t=0}$ similarly
to \eqref{eq:adjoint}. For $\dot H_{\eta_t}=L_2^0(\eta)$, however, a fixed score $a_\eta$ does not generally have $\eta_t$-mean zero so that
$a_\eta\notin\dot H_{\eta_t}$. To circumvent this problem, set $a_{\eta_t}=a-\eta_t a$ and apply the previous calculations to the score function 
$B_{\theta,\eta_t}a_{\eta_t}$ to obtain the desired result. 
\end{proof}

\end{document}